\documentclass[10pt,a4paper]{article}

\usepackage{amsmath,amsthm}                                            %
\usepackage{mathrsfs}                                           %
\usepackage{amssymb}                                          %
\newcommand{\nn}{\nonumber}                             %
\newcommand{\Ind}{1\!\mathrm{l}}                          %
\usepackage{bm}
\usepackage{graphicx}

\usepackage[top=1.37in, bottom=1.37in, left=1.1in, right=1.1in]{geometry}                                       %

\newtheorem{Theorem}{Theorem}
\newtheorem{Lemma}{Lemma}

\begin{document}
\title{Asymptotic Properties for Methods Combining Minimum Hellinger     Distance Estimates and Bayesian Nonparametric Density Estimates }

\author{ Yuefeng Wu\footnote{University of Missouri Saint Louis; wuyuel@umsl.edu} and Giles Hooker\footnote{Cornell University; gjh27@cornell.edu}}

\date{}

\maketitle

\abstract{In frequentist inference, minimizing the Hellinger distance between a kernel density estimate and a parametric family produces estimators that are both robust to outliers and statistically efficient when the parametric family is contains the data-generating distribution. This paper seeks to extend these results to the use of nonparametric Bayesian density estimators within disparity methods. We propose two estimators: one replaces the kernel density estimator with the expected posterior density using a random histogram prior;  the other transforms the posterior over densities into a posterior over parameters through minimizing the Hellinger distance for each density.   We show that it is possible to adapt the mathematical machinery of efficient influence functions from semiparametric models to demonstrate that both our estimators are efficient in the sense of achieving the Cram\'{e}r-Rao lower bound. We further demonstrate a Bernstein-von-Mises result for our second estimator, indicating that its posterior is asymptotically Gaussian. In addition, the robustness properties of classical minimum Hellinger distance estimators continue to hold.}

\section{Introduction}
This paper develops Bayesian analogs of minimum Hellinger distance methods. In particular, we aim to produce methods that enable a Bayesian analysis to be both robust to unusual values in the data and to retain their asymptotic precision when a proposed parametric model is correct.

All statistical models include assumptions which may or may not be true of the mechanisms producing a given data set. Robustness is a desired property in which a {statistical} procedure is relatively insensitive to deviations from these assumptions. For frequentist inference, concerns are largely associated with distributional robustness: the shape of  the true underlying distribution deviates slightly from the assumed model. Usually, this deviation represents the situation where there are some outliers in the observed data set{; see \cite{Huber2004} for example}.   For Bayesian procedures, the deviations may come from the model, prior distribution, or utility function, or some combination thereof. Much of the literature on Bayesian robustness has been concerned with the prior distribution or utility function.  By contrast, the focus of this paper is  robustness with respect to outliers in a Bayesian context, a relatively understudied form of robustness for  Bayesian models. For example, we know  that {Bayesian models with heavy tailed data distributions} are robust with respect to outliers for the case of one single location parameter estimated by many observations. { However, as a consequence of the Cr\'{a}mer-Rao lower bound and the efficiency of the MLE, modifying likelihoods to account for outliers will usually result in a loss of precision in parameter estimates when they are not necessary.}  The methods we propose, and the study of their robustness properties, {will provide an alternative means of making any i.i.d. data distribution robust to outliers} { that do not loose efficiency when no outliers are present.} We speculate that they can be extended beyond i.i.d. data as in \cite{Hooker11}, but do not pursue this here.

Suppose we are given the task of estimating $\theta_0 \in \Theta$ from  independent and identically distributed  {univariate random} variables $X_1,\ldots,X_n$, where we assume each $X_i$ has density $f_{\theta_0} \in \mathscr F=\{f_{\theta}:\theta\in\Theta\}$.
Within the frequentist literature,  minimum Hellinger distance estimates proceed by first estimating a kernel density $\hat g_n(x)$ and then choosing $\theta$ to minimize the Hellinger distance $h(f_{\theta},g_n)=[\int\{{f_{\theta}^{1/2}(x)} - {\hat g^{1/2}_n(x)}\}^2 dx]^{1/2}$.  The minimum Hellinger distance estimator was shown in \cite{Beran77} to have the remarkable properties of being both robust to outliers and statistically efficient, in the sense of asymptotically attaining the information bound, when the data are generated from $f_{\theta_0}$. These methods have been generalized to a class of minimum disparity estimators, based on alternative measures of the difference between a kernel density estimate and a parametric model, which have been studied since then, eg. \cite{BasuLindsay94,BSV97,PakBasu98,ParkBasu04} and \cite{Lindsay94}. {  While some adaptive M-estimators can be shown to retain both robustness and efficiency, eg. \cite{GerviniYohai02}, minimum disparity methods are the only generic methods we are aware of that retain both properties and can also be readily employed within a Bayesian context.}
In this paper, we {only consider} Hellinger distance in order to simplify the mathematical exposition; the extension to more general disparity methods can be made following similar developments to those in \cite{ParkBasu04} and \cite{BSV97}.




Recent methodology proposed in \cite{Hooker11}, suggested the use of disparity-based methods  within Bayesian inference via the construction of a ``disparity likelihood'' by replacing the likelihood function when calculating the Bayesian posterior distribution; {they} demonstrated that the resulting expected {\em  a posteriori} estimators retain the frequentist properties studied above. These methods first obtain kernel  density estimates from data and then calculate the disparity between the estimated density function and the corresponding density functions in the parametric family.

In this paper, we propose  the use of Bayesian non-parametric methods instead of the classical kernel methods in applying the minimum Hellinger distance method. The method we proposed is just to replace the kernel density estimate used in classical minimum Hellinger distance estimate by the Bayesian nonparametric expected {\em a posteriori} density, which we denote by MHB (Minimum Hellinger distance method using a Bayesian nonparametric density estimate) The second method combines the minimum Hellinger distance estimate  with the Bayesian nonparametric posterior to give a posterior distribution of the parameter of interest. This latter method is our main focus.
We show that it is more robust than usual Bayesian methods and demonstrate that  it retains asymptotic efficiency, hence the precision of the estimate is maintained. {So far as we are aware,} {this is the first Bayesian method that can be applied generically and retain both robustness and (asymptotic) efficiency.} We denote it by BHM (Bayesian inference using Minimum Hellinger distance).

To study the properties of the proposed new methods, we treat both MHB and BMH as special cases of  semi-parametric models. The general form of a semi-parametric model has a natural parametrization $(\theta, \eta)\mapsto P_{\theta, \eta}$, where $\theta\in \Theta$ is a Euclidean parameter and $\eta\in H$ belongs to an infinite-dimensional set. For such models, $\theta$ is the parameter of primary interest, while $\eta$ is a nuisance parameter. Asymptotic properties of some of Bayesian semi-parametric models have been discussed in \cite{wu09}. Our disparity based methods involve parameters in Euclidean space and  Hilbert space with the former being of most interest. However, unlike many semi-parametric models in which $P_{\theta,\eta}\in \mathcal P$ is specified jointly by $\theta$ and $\eta$, in our case the finite dimensional parameter and the parameter the nonparametric density functions are parallel specifications of the data distribution. Therefore, standard methods to study asymptotic properties of semi-parametric models will not apply to the study of disparity based methods. Nevertheless,  considering the problem of  estimating $\psi(P)$ of some function $\psi: \mathcal P \mapsto \mathbb R^d$, where $\mathcal P$ is the space of the probability models $P$,  semi-parametric models and disparity based methods can be unified into one framework.

The MHB and BMH methods are introduced in detail in Section 2 where we will also discuss some related concepts and results, such as tangent sets, information, consistency and the specific nonparametric prior that we employ.  In Section 3, both MHB and BMH are shown to be efficient, in the sense that asymptotically the variance of the estimate achieves the lower bound of the Cram\'{e}r-Rao theorem. For MHB, we  show that asymptotic normality of the estimate holds, where the asymptotic variance is the inverse of the Fisher information. For BMH, we show that a Bernstein-von Mises (BvM) theorem holds. 
The robustness property  and further discussion of these two methods are given in Section 4 and 5 respectively.


\section{Minimum Hellinger Distance Estimates}

Assume that 
random variables $X_1,\ldots,X_n$ are independent and identically distributed (iid) with density belonging to a specified parametric family $\mathscr F=\{f_{\theta}: \theta\in \Theta\}$, where all the $f_{\theta}$ in the family have the same support, denoted by $supp(f)$. For simplicity, we use $\mathbb X_n$ to denote the random variables $X_1,\ldots,X_n$.
More flexibly, we model $\mathbb X_n\sim g^n$, where $g$ is a probability density function with respect to Lebesgue measure on $supp(f)$. Let $\mathscr G$ denote the collection of all such probability density functions.
If the parametric family contains the data-generating distribution, then $g=f_{\theta}$ for some $\theta$. Formally, we can denote the probability model of the observations in the form of a semi-parametric model $(\theta, g)\mapsto P_{\theta,g}$.
We aim at estimating $\theta$ and consider $g$ as a nuisance parameter,  which is typical of semi-parametric models.



Let $\pi$ denote a prior on $\mathscr G$, and for any measurable subset $B\subset \mathscr G$, the posterior probability of $g\in B$ given $\mathbb X_n$ is
$$
\pi(B\mid \mathbb X_n)=\frac{\int_B\prod_{i=1}^ng(X_i)\pi(dg)}{\int_{\mathscr G}\prod_{i=1}^ng(X_i)\pi(dg)}.
$$
Let $g_n^*=\int g \pi(dg \mid \mathbb X_n)$ denote the Bayesian nonparametric expected {\em a posteriori} estimate. Our proposed method can be described formally as follows:

MHB:  Minimum Hellinger distance estimator with Bayesian nonparametric density estimation:
\begin{equation}\label{eq:model3}
\hat \theta_1={\rm argmin}_{\theta \in \Theta} \ h \left (f_{\theta}, {g_n^{*}} \right ).
\end{equation}
This estimator  replaces the kernel density estimate in the classical minimum Hellinger distance method introduced in \cite{Beran77} by the posterior expectation of the density function.

For this method, we will view $\hat{\theta}_1$ as the value at $g_n^*$ of a functional $T: \mathscr G \mapsto \Theta$, which is defined via
\begin{equation}\label{eq:defT}
\|f^{1/2}_{T(g)}-g^{1/2}\|=\min_{t\in\Theta}\|f_t^{1/2}-g^{1/2}\|,
\end{equation}
where $\|\cdot\|$ denotes the $L_2$ metric. We can also write $\hat \theta_1$ as $T(g_n^*)$.

In a  more general form, what we estimate is the value $\psi(P)$ of some functional $\psi: \mathcal P \mapsto \mathbb R^d$, where the $P$ stands for the common distribution from which data are generated, and  $\mathcal P$ is the set of all possible values of $P$, which also denotes the corresponding probability model. In the setting of minimum  Hellinger distance estimation, the model $\mathcal P$ is set as $\mathscr F \times \mathscr G$, $P$ can be specified as $P_{\theta, g}$, and $\psi(P)=\psi(P_{\theta,g})=\theta$. For the methods we proposed in this paper we will focus on the  functional $T: \mathscr G\mapsto \Theta$, for a given $\mathscr F$, as defined above. Note that the constraint associated to the family $\mathscr F$ is implicitly applied by $T$.

Using functional $T$, we can also propose a Bayesian method, which assigns nonparametric prior on the density space and gives inference on the unknown parameter $\theta$ of a parametric family as follows:

BMH: Bayesian inference with minimum Hellinger distance estimation:
\begin{equation}
\pi(\theta\mid \mathbb X_n)=\pi(T(g)\mid \mathbb X_n).
\end{equation}
A nonparametric prior $\pi$  on the space $\mathscr G$  and the observation $\mathbb X_n$ leads to the posterior distribution $\pi(g\mid \mathbb X_n)$, which can then be converted to the posterior distribution of  the parameter $\theta \in \Theta$ through the functional $T: \mathscr G \mapsto \Theta$.

In the following subsections, we  discuss properties associated with the functional $T$, the consistency of MHB and BHM, and give a detailed example of the random histogram prior that we will employ, and its properties that will be used for the discussion of efficiency in the next section.

\subsection{Tangent Space and Information}

In this subsection, we obtain the efficient influence function of the functional $T$ on the linear span of the tangent set on $g_0$,   and show that the local asymptotic normality (LAN) expansion related to the norm of the efficient influence function attains the Caram\'{e}r-Rao bound. These results play  important roles in showing that BvM holds for BMH method in the next section.

Estimating the parameter by $T(g)$ under the assumption $g\in \mathscr G$ uses less information than estimating this parameter for $g\in \mathscr G^* \subset \mathscr G$. Hence the lower bound of the variance of  $T(g)$ for $g\in \mathscr G$ should be at least the supremum of the lower bounds of all parametric sub-models $\mathscr G^*=\{ G_{\lambda}: \lambda\in \Lambda \}\subset \mathscr G$.

To use mathematical tools such as functional analysis to study the properties of the proposed methods, we introduce some notations and concepts below.
Without loss of generality, we consider one-dimensional sub-models $\mathscr G^*$, which pass through the ``true'' distribution, denoted by $G_0$ with density function $g_0$. 
We say a sub-model indexed by $t$, $\{g_t:0<t<\epsilon\}\subset \mathscr G$,  is differentiable in quadratic mean at $t=0$ if we have that, for some measurable function $q: supp(g_0) \mapsto \mathbb R$,
\begin{equation}\label{eq:25.13}
\int\left [  \frac{dG_t^{1/2}-dG_0^{1/2}}{t}-\frac{1}{2} qdG_0^{1/2} \right]^2\to 0,
\end{equation}
where $G_t$ is the cumulative distribution function  associated to $g_t$.
Functions $q(x)$s are  known as the score functions associated to each sub-model.  The collection of these score functions,  which is called a tangent set of the model $\mathscr G$ at $g_0$ and denoted by $\dot {\mathscr G}_{g_0}$, is induced by the collection of all  sub-models that are differentiable at $g_0$.

We say that $T$ is differentiable at $g_0$ relative to a given tangent set $\dot{\mathscr G}_{g_0}$,
if there exists a continuous linear map $\dot {T}_{g_0}: L_2(G_0)\mapsto \mathbb R$ such that for every $q\in \dot{\mathscr G}_{g_0}$ and a sub-model $t\mapsto g_t$ with score function $q$,  there is
\begin{equation}
\frac{T(g_t)-T(g_0)}{t}\to \dot{T}_{g_0} q,
\end{equation}
where $L^2(G_0)=\{q: supp(g_0)\mapsto \mathbb R, \int q^2(x)g_0(x)dx<\infty\}$.
By the Riesz representation theorem for Hilbert spaces, the map $\dot {T}_{g_0}$ can always be written in the form of an inner product with a fixed vector-valued, measurable function $\tilde {T}_{g_0}: supp(g_0) \mapsto \mathbb R$,
\begin{equation*}
\dot {T}_{g_0} q= \langle \tilde {T}_{g_0}, q \rangle _{G_0}=\int \tilde {T}_{g_0} q dG_0.
\end{equation*}
Let $\tilde {T}_{g_0}$ denote the unique function in $\overline{lin} \dot{\mathscr G}_{g_0}$, the closure of the linear span of the tangent set. The function $\tilde {T}_{g_0}$ is the efficient influence function and can be found as the projection of any other ``influence function'' onto the closed linear span of the tangent set.

For a sub-model $t\mapsto g_t$ whose score function is $q$, the Fisher information about $t$
at $0$ is $G_0 q^2=\int q^2dG_0$, and in this paper we use the notation $Fg$  to denote $\int gdF$ for a general function $g$ and distribution $F$. Therefore, the ``optimal asymptotic variance'' for estimating the functional $t\mapsto T(g_t)$, evaluated at $t=0$, is greater than or equal to the Caram\'er-Rao bound
\begin{equation*}
\frac{(dT(g_t)/dt)^2}{G_0 q^2}=\frac{\langle \tilde{T}_{g_0}, q \rangle^2_{G_0}}{\langle  q,q \rangle_{G_0}}.
\end{equation*}
The supremum of the right hand side (RHS) of the above expression over all elements of the tangent set is a lower bound for estimating $T(g)$ given model $\mathscr G$, if the true model is $g_0$. The supremum can be expressed in the norm of the efficient influence function $\tilde {T}_{g_0}$, by Lemma 25.19 in \cite{van2000asymptotic}. The lemma and its proof is quite neat and we reproduce it here for the completeness of the argument.

\begin{Lemma}\label{lemma25.19}
	Suppose that the functional $T: \mathscr G \mapsto \mathbb R$ is differentiable at $g_0$ relative to the tangent set $\dot {\mathscr G}_{g_0}$. Then
	\begin{equation*}
	\sup_{q\in lin \dot{\mathscr G}_{g_0}}   \frac {\langle  \tilde{T}_{g_0}, q \rangle^2_{G_0}}  {\langle q,q \rangle_{G_0}} = G_0 \tilde{T}_{g_0}^2.
	\end{equation*}
\end{Lemma}
\begin{proof} This is a consequence of the Cauchy-Schwarz inequality $(G_0\tilde{T}_{g_0}q )^2\leq G_0\tilde {T}_{g_0}^2 G_0q^2$ and the fact that, by definition, the efficient influence function, $\tilde T_{g_0}$, is contained in the closure of $lin \dot{\mathscr G}_{G_0}$.
\end{proof}

Now we show that functional $T$ is differentiable under some mild conditions and construct its efficient influence function in the following theorem.
\begin{Theorem}\label{lem:tildeT}
	For the functional $T$ defined in (\ref{eq:defT}),  and
	for $t\in \Theta\subset \mathbb R$, let $s_t(x)$ denote $f^{1/2}_{\theta}(x)$ for $\theta=t$, we assume that there exist   $\dot s_t(x)$ and   $\ddot{s}_t(x)$ both in $L_2$,  such that for  $\alpha$ in a neighborhood of zero,
	\begin{eqnarray}
	s_{t+\alpha }(x) &=&s_t(x)+\alpha \dot {s} _t(x)+\alpha u_{\alpha} (x), \label{eq:2.5}\\
	\dot {s}_{t+\alpha }(x) &=&\dot{s}_t(x)+ \alpha \ddot {s}_t(x)+\alpha v_{\alpha} (x) \label{eq:2.6},
	\end{eqnarray}
	where $u_{\alpha}$ and $v_{\alpha}$ converge to zero as $\alpha \to 0$. Assuming  $T(g_0)\in  int(\Theta)$,  the efficient influence function of  $T$ is
	\begin{equation}\label{eq:efficientfunc}
	\tilde T_{g_0}=  \left ( -\left [\int  \ddot s_{T(g_0)}(x) g_0^{\frac{1}{2}}(x)dx\right ] ^{-1}   +a_t\right)
	\frac{\dot s_{T(g_0)}(x)  } {2g_0^{\frac{1}{2}(t)}} 
	\end{equation}
where $a_t$  converges to $0$ as $t\to 0$. In particular, for $g_0=f_{\theta}$
\begin{equation}\label{eq:efficientfunc2}
\tilde T_{f_{\theta}}=  \left ( -\left [\int  \ddot s_{\theta}(x) s_{\theta}(x)dx\right ] ^{-1}   +a_t\right)
\frac{\dot s_{\theta}(x)  } {2s_{\theta}(x)}
\end{equation}
\end{Theorem}

\begin{proof} Let the $t$-indexed sub-model be
\begin{equation}\label{eq:maxtan}\nn
g_t:=(1+tq(x))g_0(x),
\end{equation} where $q(x)$ satisfies $\int q(x)g_0(x)dx=0$ and $q\in L_2(g_0)$.  By direct calculation, we see that $q$ is the score function associated to such sub-model at $t=0$ in the sense of (\ref{eq:25.13}) and thus the collection of $q$ is the maximal tangent set.

By the definition of $T$,  $T(g_0)$ maximizes $\int s_t(x)g_0^{1/2}(x)dx$.  From (\ref{eq:2.5}), we have that
\begin{equation}
lim_{\alpha\to 0} \alpha^{-1}\int [ s_{t+\alpha }(x) -s_t(x) ]g_0^{1/2}(x)dx=\int \dot s_{t}(x)g_0^{1/2}(x)dx.
\end{equation}
Since $T(g_0)\in int(\Theta)$, we have that
\begin{equation}\label{eq:inner0}
\int \dot s_{T(g_0)}(x)g_0^{1/2}(x)dx=0.
\end{equation}
Similarly  $\int \dot s_{T(g_t)}(x)g_t^{1/2}(x)dx  =0$.  Using (\ref{eq:2.6}) to substitute $\dot s_{T(g_t)}$, we have that
$$
0= \int [ \dot s_{T(g_0)}(x)+\ddot s _{T(g_0)}(x)(T(g_t)-T(g_0))+v_t(x)(T(g_t)-T(g_0)) ]g_t^{1/2}(x)dx,
$$
where the components of the $p\times p$ matrix $v_t(x)$ converge in $L_2$ to zero as $t\to 0$ since $T(g_t)\to T(g_0)$. Thus,
\begin{eqnarray}
\lefteqn {\lim_{t\to 0}\frac{1}{t}[T(g_t)-T(g_0)]}\nn\\
&&= -\lim_{t\to 0}\frac{1}{t}  \left [\int ( \ddot s_{T(g_0)}(x) +v_t(x)  )g_t^{\frac{1}{2}}(x)dx\right ] ^{-1} \int \dot s_{T(g_0)}(x)g_t^{1/2}(x)dx \nn\\
&&=  \lim_{t\to 0}\frac{1}{t}   \left ( -\left [\int ( \ddot s_{T(g_0)}(x) )g_0^{\frac{1}{2}}(x)dx\right ] ^{-1}   +a_t\right)  \int \dot s_{T(g_0)}(x)(g_t^{\frac{1}{2}}(x)-g_0^{\frac{1}{2}}(x))dx \nn \\
&&=       \left ( -\left [\int ( \ddot s_{T(g_0)}(x) )g_0^{\frac{1}{2}}(x)dx\right ] ^{-1}   +a_t\right)  \int \frac{\dot s_{T(g_0)}(x)  } {2g_0^{\frac{1}{2}}(x)} q(x)g_0(x)dx  \nn
\end{eqnarray}
Since by the definition of $\tilde T$, which requires $\int \tilde T_{g_0} g_0(x)dx=0$, we have that
\begin{eqnarray*}
	\tilde T_{g_0}&= & \left ( -\left [\int  \ddot s_{T(g_0)}(x) g_0^{\frac{1}{2}}(x)dx\right ] ^{-1}   +a_t\right)
	\left(  \frac{\dot s_{T(g_0)}(x)  } {2g_0^{\frac{1}{2}}(x)} -
	\int \frac{\dot s_{T(g_0)}(x)  } {2}g_0^{\frac{1}{2}}(x)dx
	\right )\\
	&=& \left ( -\left [\int  \ddot s_{T(g_0)}(x) g_0^{\frac{1}{2}}(x)dx\right ] ^{-1}   +a_t\right)
	\frac{\dot s_{T(g_0)}(x)  } {2g_0^{\frac{1}{2}}(x)} 
\end{eqnarray*}
By the same argument we can show that when $g_0=f_{\theta}$, equation (\ref{eq:efficientfunc2}) holds.
\end{proof}

Some relatively accessible conditions under which (\ref{eq:2.5}) and (\ref{eq:2.6}) hold are given  by Lemma 1 and 2 in \cite{Beran77}. We do not repeat them here.

Now we can expand $T$ at $g_0$ as
\begin{equation}\label{eq:4.2}
T(g)-T(g_0)=\langle \frac{g-g_0}{g_0}, \tilde T_{g_0} \rangle_{G_0}+\tilde r(g,g_0),
\end{equation}
where $\tilde T$ is given in Theorem \ref{lem:tildeT} and $\tilde r=0$.


\subsection{Consistency of MHB and BMH}

Since $T(g)$ may have more than one value, the notation $T(g)$ is used to denote any arbitrary one of the possible values. In \cite{Beran77}, the existence, continuity in Hellinger distance and uniqueness of functional $T$ are ensured under the conditions:

\begin{description}
\item[A1] (i) $\Theta$ is compact, (ii) $\theta_1\neq \theta_2$ implies $f_{\theta_1}\neq f_{\theta_2}$ on a set of positive Lebesgue measure, and (iii) for almost every $x$, $f_{\theta}(x)$ is continuous in $\theta$. \label{cond:a1}

When a Bayesian nonparametric denstity estimatar is used, we assume the posterior  consistency:

\item[A2] For any given $\epsilon>0$, $\pi\{g: h(g,f_{\theta_0})>\epsilon\mid \mathbb X_n  \}\to 0$ in probability. \label{cond:a2}
\end{description}
Under conditions A1 and A2,  consistency holds for MHB and BMH.

\begin{Theorem}\label{thm:1}
	Suppose that conditions A1 and A2 hold, then
	\begin{itemize}
		\item [1.]  $\| g_n ^{*1/2}-f_{\theta_0}^{1/2}\|^2\to 0$ in probability,   $T(g_n^{*})\to T(f_{\theta_0})$ in probability, and hence $\hat\theta_1 \to \theta_0$ in probability;
		\item [2.] For any given $\epsilon>0$, $\pi(|\theta-\theta_0|>\epsilon\mid \mathbb X_n) \to 0 $ in probability.
	\end{itemize}
\end{Theorem}

\begin{proof} Part 1:
To show that $\|g_n^{*1/2}-f_{\theta_0}^{1/2}\|^2\to 0$ in probability, which is equivalent to showing that $\int \left ( \int g\pi(dg\mid \mathbb X_n) ^{1/2}-f^{1/2}_{\theta_0}\right)^2 dx \to 0 $ in probability, it is sufficient to show that $\int \left| \int g\pi(dg\mid \mathbb X_n)-f_{\theta_0}\right|dx \to 0$ in probability, since $h^2(f,g)\leq \|f-g\|_1$. We have that
\begin{eqnarray*}
	\int \left| \int g\pi(dg\mid \mathbb X_n)-f_{\theta_0}\right|dx&=&
	\int \left| \int (g-f_{\theta_0})\pi(dg\mid \mathbb X_n)\right|dx \\
	& \leq & \int  \int \left|g-f_{\theta_0}\right|\pi(dg\mid \mathbb X_n)dx\\
	&=&\int  \int \left|g-f_{\theta_0}\right|dx \pi(dg\mid \mathbb X_n) \\
	& \leq & \int  \sqrt 2 h(g,f_{\theta_0}) \pi(dg\mid \mathbb X_n).
\end{eqnarray*}
Note that the change of order of integration is due to Fubini's theorem and the last inequality is due to  $\|f-g\|_1\leq \sqrt 2 h(f,g)$.
Split the integral on the right hand side of the above expression into two parts:
$$\int_{\mathscr A}  \sqrt 2h(g,f_{\theta_0}) \pi(dg\mid \mathbb X_n) + \int_{\mathscr A^c}  \sqrt 2h(g,f_{\theta_0}) \pi(dg\mid \mathbb X_n),$$
where $\mathscr A=\{g: h(g,f_{\theta_0})\leq \epsilon \}$ for any given $\epsilon>0$. The first term is bounded by $\epsilon$ by construction. By condition A1, the posterior of measure of $\mathscr A^c$ to $0$ in probability as $n \to \infty$. Since Hellinger distance is bounded by 2, so does the second term above.  This completes the proof for $\| g_n ^{*1/2}-f_{\theta_0}^{1/2}\|^2\to 0$ in probability.

To show $T(g_n^{*})\to T(f_{\theta_0})$ and  $\hat\theta_1 \to \theta_0$ in probability, we need that the functional $T$ is continuous and unique at $f_{\theta_0}$, which is proved by Theorem 1 in \cite{Beran77} under condition A1.

Part 2:  By condition A1 and  Theorem 1 in \cite{Beran77},  the functional $T$ is continuous and unique at $f_{\theta_0}$. Hence, for any given $\epsilon>0$, there exist $\delta>0$ such that $|T(g)-T(f_{\theta_0})|<\epsilon$ when $h(g,f_{\theta_0})<\delta$. By condition A2, we have that $\pi(h(g,f_{\theta_0})<\delta)\to 1$, which implies that $\pi(|\theta-\theta_0|<\epsilon)\to 1$ in probability.
\end{proof}

Note that if we change the $\epsilon$ in condition A2 to $\epsilon_n$, a sequence converging to $0$, then we can apply the results for the concentration rate of the Bayesian nonparametric density estimation here. However, such approach cannot lead to the general ``efficiency'' claim, no matter in the form of rate of concentration or asymptotic normality. There are two reasons for this. First, the rate of concentration for Bayesian nonparametric posterior is about $n^{-2/5}$ for a rather general situation and $(\log n)^a \times n^{-1/2}$, where $a>0$, for some special cases (see \cite{ghosal00},  \cite{ghosal07}, \cite{Ghosh2003}).
This concentration rate is not sufficient in many situations to  directly imply that the concentration of the corresponding parametric estimates achieve the lower bound of the variance  given in the Cram\'{e}r-Rao theorem. Second, the Hellinger distances between pairs of densities as  functions of  parameters, vary among different parametric families. Therefore, obtaining the rate of concentration in parameters from the rate of convergence in the densities cannot be generally applied to different distribution families. 

Also note that, although $\Theta$ is required to be compact in condition A1., Theorem \ref{thm:1} is useful for $\Theta$ that is not compact, as long as the parametric family $f_{\theta}: \theta\in \Theta$ can be re-parameterized where the space of new parameters can be embedded within a  compact set. An example of re-parameterizing a general location-scale family with parameters $\mu\in \mathbb R$ and $\sigma\in \mathbb R^+$ to a family  with parameters $t_1=tan^{-1}(\mu)$ and $t_2=tan^{-1}(\sigma)$, where $\Theta_{(t_1,t_2)}=(-\pi/2, \pi/2) \times (0,\pi/2)$ and $\Theta \subset \bar \Theta = [-\pi/2, \pi/2] \times [0,\pi/2]$, is discussed in \cite{Beran77}, and the conclusions of Theorem 1 in \cite{Beran77} is  still valid for a location-scale family. Therefore, Theorem \ref{thm:1} remains valid for the same type of the families, whose parameter space may not be compact  and for the same reasons; the compactness requirement stated in the theorem is mainly for the mathematical simplicity.

\subsection{Prior on Density Functions}

We introduce a random histogram as an example for priors used in Bayesian nonparametric density estimation. It can be seen as a simplified version of Dirichlet process mixture (DPM) prior, which is commonly used in practice. Both DPM and random histogram are mixture densities. While DPM uses a Dirichlet process to model the weights within an infinite mixture of kernels, the random histogram prior only has finite number of components. Another difference is that although we specify the form of the kernel function for DPM, the kernel function could be any density function in general, while the random histogram uses only the uniform density as it mixing kernel. Nevertheless, the limit on finite number of the mixing components is not that important in practice, since the Dirichlet process will  always be truncated in computation.  In next section, we will verify that the random histogram satisfies the conditions that are needed for our proposed methods to be efficient. On the other hand, although we believe that DPM should also lead to efficiency, the authors are unaware of the  theoretical results or tools required to prove it. This is mostly due to the flexibility of DPM, which in turn significantly increases the mathematical complexity of the analysis.


For any $k \in \mathbb N$, denote the set of all regular $k$ bin histograms on $[0,1]$ by
$
\mathscr H_k=\{f\in L^2([0,1]): m(x)=\sum_{j=1}^{k}f_j\Ind_{I_j}(x), f_j\in \mathbb R, j=1,\ldots,k \},
$
where $I_j=[ (j-1)/k,j/k )$.  Denote the unit simplex in $\mathbb R^k$ by
$\mathscr S_k=\{\omega \in [0,1]^k: \sum _{j=1}^k\omega_j=1 \}$
The subset of $\mathscr H_k$,
$
\mathscr H_k^1=  \{ f\in L^2(\mathbb R), f(x)=f_{\omega, k}=k\cdot\sum_{j=1}^k \omega_j\Ind_{I_j}(x), (\omega_1, \ldots, \omega_k)\in \mathscr S_k \}$, denotes the collection of densities on $[0,1]$ in the form of histogram.


The set $\mathscr H_k$ is a closed subset of $L_2[0,1]$. For any function $f\in L_2[0,1]$, denote its projection in $L_2$ sense on $\mathscr H_k$ by $f_{[k]}$, where $f_{[k]}=k\sum_{j=1}^k \Ind_{I_j}\int_{I_j}f$.


We assign priors on $\mathscr H_k^1$ via $k$ and $(\omega_1,\ldots,\omega_k)$ for each $k$.  A degenerate case is to let $k=K_n=o(n)$. Otherwise, let $p_{k}$ be a distribution on positive integers, where
\begin{equation}\label{eq:4.6}
k\sim p_{k}, e^{-b_1k\log(k)}\leq p_{k}(k)\leq e^{-b_2k\log(k)},
\end{equation}
for all $k$ large enough and some $0<b_1<b_2<\infty$. For example, condition (\ref{eq:4.6}) is satisfied by the Poisson distribution, which is commonly used in Bayesian nonparametric models.

Conditionally on $k$, we consider a Dirichlet prior on $\omega=\{\omega_1,\ldots, \omega_k \}$:
\begin{equation}\label{eq:4.5}
\omega\sim \mathscr D(\alpha_{1,k}, \ldots, \alpha_{k,k}),\qquad c_1k^{-a}\leq \alpha_{j,k}\leq c_2,
\end{equation}
for some fixed constants $a, c_1, c_2>0$ and any $1\leq j\leq k$. For posterior consistency, we need the following condition:
\begin{equation}\label{eq:4.9}
\sup_{k\in \mathscr K_n} \sum _{j=1}^{k} \alpha_{j,k} = o(\sqrt n),
\end{equation}
where $\mathscr K_n \subset \{1,2,\ldots, \lfloor n/(\log n)^2 \rfloor \}$.

The consistency result of this prior is given by Proposition 1. in the supplement to \cite{castillo2015bernstein}.  For $n\geq 2, k\geq 1, M>0$,   let
\begin{equation}\label{eq:4.8}
A_{n,k}(M)=\{g\in \mathscr H_k^1, h(g, g_{0, [k]})\}<M\epsilon_{n,k}
\end{equation}
where $\epsilon_{n,k}^2=k\log n/n$, denote a neighborhood of $g_{0,[k]}$, we have that
\begin{itemize}
	\item (a) there exist $c$, $M>0$ such that
	\begin{equation}\label{eq:prior1}
	P_0\left [  \exists k\leq \frac{n}{\log n} ; \pi[g \notin A_{n,k}(M)\mid \mathbb X_n, k]> e^{-ck\log n} \right ]=o(1).
	\end{equation}
	\item (b) Suppose $g_0\in \mathscr C^{\beta}$ with $0<\beta\leq 1$. If $k_n(\beta)=(n\log n)^{1/(2\beta+1)}$ and $\epsilon_n(\beta)=k_n(\beta)^{-\beta}$, then for $k_1$, $M$ sufficiently large,
	\begin{equation}\label{eq:prior2}
	\pi[h(g_0,g)\leq M\epsilon_n(\beta); k\leq k_1k_n(\beta)\mid \mathbb X_n]=1+o_p(1),
	\end{equation}
	where $\mathscr C^{\beta}$ denotes the class of $\beta$-H\"{o}lder functions on $[0,1]$.
\end{itemize}
This means that the posterior of the density function  concentrates around the projection $g_{0[k]}$ of $g_0$ and also around $g_0$ itself in terms of the Hellinger distance. We can easily conclude that
$\pi(\mathscr K_n \mid \mathbb X_n)=1+o(1)$  from (\ref{eq:prior2}) for $g_0\in \mathscr C^{\beta}$.

Note that although the priors we defined above are on the densities on $[0,1]$, this is for mathematical simplicity, which could easily be extended to space of probability densities on any given compact set. Further, transformations of $\mathbb X_n$, similar to those discussed at the end of Subsection 2.2, can extend the analysis to the real line; also refer to \cite{Beran77, Amewou2003} for more example and details.

\section{Efficiency}

We say that both MHB and BMH methods are efficient if the lower bound of the variance of the estimate, in the sense of Cram\`{e}r and Rao's theorem,  is achieved.


\subsection{Asymptotic Normality of MHB}
Consider the maximal tangent set at $g_0$, which is defined as $\mathcal{H}_T=\{q\in L^2(g_0), \int q g_0=0 \}$. Denote the inner product on $\mathcal H_T$ by $\langle q_1, q_2\rangle_L=\int q_1q_2 g_0$, which induces the L-norm as:
\begin{equation}\label{eq:4.1-1}
\|g\|_L^2=\int_0^1 (g-G_0g)^2g_0.
\end{equation}
Note that the inner product $\langle \cdot, \cdot\rangle_L$ is equivalent to the inner product introduced in Section 2.1, and the induced L-norm  corresponds to the local asymptotic normality (LAN) expansion. Refer to \cite{Rivoirard2012} and Theorem 25.14 in \cite{van2000asymptotic} for more details.

With functional $T$ and priors on $g$ defined  in previous section,  Theorem \ref{thm:2} shows that MHB method is efficient when the parametric family contains the true model.
\begin{Theorem}\label{thm:2}
	Let two priors $\pi_1$ and $\pi_2$ be defined by (\ref{eq:4.6})-(\ref{eq:4.5}) and prior on $k$ be either a Dirac mass at $k=K_n=n^{1/2}(\log n)^{-2}$ for $\pi_1$ or $k\sim \pi_k$ given by (\ref{eq:4.6}) for $\pi_2$. Then the limit distribution of $n^{1/2}[T(g^*_n)-T(g_0)]$ under $g_0$ as $n\to \infty$ is $Norm(0, \|\tilde T_{g_0}\|^2_L)$, where $\|\tilde T_{g_0}\|^2_L=I(\theta_0)^{-1}$ when $g_0=f_{\theta_0}$.
\end{Theorem}

\begin{proof} To prove this result, we  verify Lemma 25.23 in \cite{van2000asymptotic}, which is equivalent to showing that $$\sqrt n (T(g^*_n)-T(g_0))=\frac{1}{\sqrt n}\sum_{i=1}^{n} \tilde T_{g_0}(X_i)+o_{p}(1).$$
By the consistency result provided for prior $\pi_1$ and $\pi_2$ in the previous section, we consider only $g_n^*\in A_{n,k}$ for $n$ sufficiently large.  Then
by equation (\ref{eq:4.2}), we have that
$$
\sqrt n(T(g^*_n)-T(g_0))=\sqrt n \left\langle \frac {g_n^*-g_0}{g_0}, \tilde T_{g_0}\right\rangle_L+o_p(1).
$$
Therefore, showing
$$\sqrt n \int _0^1 (g^*_n(x)-g_0(x))\tilde T_{g_0}(x)dx=\frac{1}{\sqrt n}\sum_{i=1}^{n} \tilde T_{g_0}(X_i)+o_{p}(1)$$
will complete the proof. Due to $\int _0^1 g_0(x)\tilde T_{g_0}(x)dx=0$, we now need to show that $\int _0^1 g_n^*(x)\tilde T_{g_0}(x)dx=(1/n)\sum_{i=1}^{n} \tilde T_{g_0}(X_i)+o_{p}(1)$.
By the law of large numbers, we have that
$\frac{1}{n}\sum_{i=1}^{n} \tilde T_{g_0}(X_i)-G_0 \tilde T_{g_0}=o_p(1)$, and $\int _0^1 g_n^*(x)\tilde T_{g_0}(x)dx
-G_0 \tilde T_{g_0}=o_p(1)$ due to the posterior consistency demonstrated above. Therefore, we have that
\begin{eqnarray*}
\lefteqn{\frac{1}{n}\sum_{i=1}^{n} \tilde T_{g_0}(X_i)-\int _0^1 g_n^*(x)\tilde T_{g_0}(x)dx} \\
&=&\frac{1}{n}\sum_{i=1}^{n} \tilde T_{g_0}(X_i)
-\int _0^1 g_n^*(x)\tilde T_{g_0}(x)dx
+\int _0^1 g_n^*(x)\tilde T_{g_0}(x)dx
-G_0 \tilde T_{g_0} \\
&=& o_p(1).
\end{eqnarray*}
\end{proof}


\subsection{Bernstein-von Mises Theorem for BMH}
Theorem 2.1 in \cite{castillo2015bernstein} gave a general result and approach to show the BvM Theorem holds for smooth functionals in some semi-parametric models. The theorem shows that under the continuity and consistency condition, the moment generating function (MGF) of the parameter endowed with posterior distribution can be calculated approximately  through the local asymptotic normal (LAN) expansion, its convergence to an MGF of some normal random variable then can be shown under some assumptions on the statistical model.

We will show that BvM theorem holds for BMH Method via Theorem \ref{thm:4.2}.  The result also shows  that the approach given in \cite{castillo2015bernstein} can be applied  not only to simple examples, but also to relatively complicated frameworks. To prove it, we introduced Lemma \ref{lemma:prop1}, which is modified from Proposition 1 in \cite{castillo2015bernstein}, the proof of which was not given explicitly in the original paper.

For mathematical simplicity, we assume that the true density $f_{\theta_0}$ belongs to the set $\mathscr F$, which is restricted to the space of all densities that are bounded away from $0$ and $\infty$ on $[0,1]$. As noted above, the compactness of the domain can be relaxed by considering transformations of the parameters and random variables.

To state the Lemma, we need several more notations.
Assume that the functional $T$ satisfies (\ref{eq:4.2}) with bounded efficient influence function $\tilde T_{g_0} \neq 0$, we denote $\tilde T_{g_0}$ by $\tilde T$, where $\tilde T_{[k]}$ denotes the projection of $\tilde T$ on $\mathscr H_k$. For $k\geq 1$, let

\begin{eqnarray}\label{eq:4.7}
\hat T_k&=& T(g_{0[k]})+\frac{\mathbb G_n \tilde T_{[k]}}{\sqrt n}, \qquad V_k=\|\tilde T_{[k]}\|_L^2,\nn\\
\hat T&=& T(g_0)+\frac{\mathbb G_n \tilde T}{\sqrt n}, \qquad V=\|\tilde T\|_L^2,
\end{eqnarray}
 and denote
\begin{equation}\label{eq:4.1-2}
\mathbb G_n(g)=W_n(g)=\frac{1}{\sqrt n}\sum_{i=1}^n[g(x_i)-G_0(g)].
\end{equation}

\begin{Lemma}\label{lemma:prop1}
	Let $g_0$ belong to $\mathscr G$,  the prior $\pi$ be defined as in section 2.3, and conditions (\ref{eq:4.6}, \ref{eq:4.5}, \ref{eq:4.9}) be satisfied.
	Consider estimating a  functional $T(g)$, differentiable with respect to the tangent set $\mathscr H_T:=\{q\in L^2(g_0), \int_{[0,1]} qg_0=0\}\subset \mathscr H=L^2(g_0)$, with efficient influence function $\tilde T_{g_0}$ bounded on $[0,1]$, and with $\tilde r$ defined in (\ref{eq:4.2}), for $\mathscr K_n$ as introduced in (\ref{eq:4.9}). If
	\begin{equation}\label{eq:4.10-1}
	\max_{k\in \mathscr K_n}\left | \|\tilde T_{[k]}\|_L^2 - \|\tilde T\|_L^2\right |=o_p(1),
	\end{equation}
	
	\begin{equation}\label{eq:4.10-2}
	\max_{k\in \mathscr K_n}\mathbb G_n (\tilde T_{[k]}-\tilde T)=o_p(1),
	\end{equation}
	\begin{equation}\label{eq:4.11}
	\sup_{k\in \mathscr K_n} \sup_{g\in A_{n,k}(M)} \sqrt n \tilde r(g,g_0)=o_p(1),
	\end{equation}
	for any $M>0$ and $A_{n,k}(M)$ defined as in (\ref{eq:4.8}), as $n\to \infty$,  and
	\begin{equation}\label{eq:4.12}
	\max_{k\in\mathscr K_n} \sqrt n \left |  \int (\tilde T- \tilde T_{[k]})(g-g_0)\right |=o(1),
	\end{equation}
	then the BvM theorem for the functional $T$ holds.
\end{Lemma}



\begin{proof}
To show that BvM holds is to show that the  posterior distribution converges to a normal distribution. If we have that
\begin{eqnarray}\label{eq:normal}
{\pi[\sqrt n(T-\hat T_k)\leq z \mid \mathbb X_n]}
&=&\sum_{k\in \mathscr K_n} \pi[k \mid \mathbb X_n]\pi\left[\sqrt n(T-\hat T)\leq z+\sqrt n (\hat T-\hat T_k)\mid \mathbb X_n, k\right ]+o_p(1)\nn\\
&=&\sum_{k\in \mathscr K_n}\pi[k \mid \mathbb X_n] \Phi \left(\frac{z+\sqrt n (\hat T-\hat T_k)}{\sqrt {V_k}}\right)+o_p(1),
\end{eqnarray}
then the proof will be completed by showing that the R.H.S. of equation (\ref{eq:normal})  reduces from the mixture of normal to the target law $N(0,V)$.

By condition (\ref{eq:4.10-1}), we have that $V_k$ goes to $V$ uniformly for $k \in \mathscr K_n $. Due to the definition of $\tilde T$ and the Lemma 4 result (iii) in the supplement of \cite{castillo2015bernstein}, we have that
\begin{eqnarray*}
	\sqrt n (\hat T-\hat T_k) &=& \sqrt n\left (T(g_0)-T(g_{[k]})\right)+\mathbb G_n(\tilde T-\tilde T_{[k]}) \\
	& = & \sqrt n \int \tilde T (g_{0[k]}-g_0)+\mathbb G_n(\tilde T_{[k]}-\tilde T)+o_p(1)\\
	&=& \sqrt n \int (\tilde T- \tilde T_{[k]})(g_{0[k]}-g_0)+\mathbb G_n(\tilde T_{[k]}-\tilde T)+o_p(1).
\end{eqnarray*}
By Condition (\ref{eq:4.12}) and (\ref{eq:4.10-2}), the last line converges to $0$ uniformly for $k\in \mathscr K_n$.

Therefore,   showing that for any given $k$,  equation (\ref{eq:normal}) holds will complete the proof.  We prove this by showing that the MGF (Laplace transformation) of the posterior distribution of the parameter of interest converges to the MGF of some normal distribution, which implies that the posterior converges to the normal distribution weakly by Lemma 1 and 2 in supplement to  \cite{castillo2015bernstein} or Theorem 2.2 in \cite{Bagui2016}.

First, consider the deterministic $k=K_n$ case. We calculate the MGF as:
\begin{equation}\label{eq:I}
 {E[ e^{t\sqrt n (T(g)-\hat T (g_{0[k]}))}  \mid \mathbb X_n, A_n]}
=\frac{\int_{A_n}e^{t\sqrt{n}(T(g)-\hat T(g_{0[k]}))+{l}_n(g)-{l}_n(g_{0[k]})} d\pi(g)}
{\int_{A_n}e^{{l}_n(g)-{l}_n(g_{0[k]})} d\pi(g)},
\end{equation}
where $l_n(g)$ is the log-likelihood for given $g$ and $\mathbb X_n$. Based on the LAN expansion of the log-likelihood and the smoothness of the functional, the exponent in the numerator on the  RHS of the equation can be transformed with respect to $\bar T_{(k)}=\tilde T_{[k]}-\int \tilde T_{[k]}g_{0[k]}$,
\begin{eqnarray*}
	\lefteqn
	{t\sqrt n(T(g)-\hat T_k)+l_n(g)-l_n(g_{0[k]})}\\
	&=& t\sqrt n\left ( T(g)-T(g_{0[k]})-\frac{\mathbb G_n \tilde T_{[k]}} {\sqrt n} \right )+l_n(g)-l_n(g_{0[k]})\\
	&=&t\sqrt n\left ( \left \langle \log \frac{g}{g_{0[k]}}-\int \log \frac{g}{g_{0[k]}}g_{0[k]}, \bar T_{[k]} \right \rangle_L+\mathscr B(g,g_{0[k]})+\tilde r(g,g_{0[k]})
	-\frac{\mathbb G_n \bar T_{[k]}} {\sqrt n} \right ) \\
	&& \hspace{0.5cm} -\frac{1}{2}\left\|\sqrt n \log \frac{g}{g_{0[k]}} \right\|_L^2+W_n\left (   \sqrt n \log \frac{g}{g_{0[k]}} \right )+R_{n,k}(g,g_{0[k]}),
\end{eqnarray*}
where $\mathscr B(g,g_0)=\int_0^1[\log(g/g_0)-(g-g_0)/g_0](x)\tilde T_{g_0}(x)g_0(x)dx$.
Note that $\mathbb G_n= W_n$, add a term of $({t^2}/{2})\|\bar T_{(k)}\|_L^2$,  then re-arranging the RHS expression above we have that
\begin{eqnarray*}
	\lefteqn
	{t\sqrt n(T(g)-\hat T_k)+l_n(g)-l_n(g_{0,[k]})}\\
	&=& -\frac{n}{2}\left \| \log \frac{g}{g_{0[k]}}-\frac{t}{\sqrt n}\bar T_{(k)} \right \|_{L,k}^2
	+\sqrt nW_n \left (
	\log \frac{g}{g_{0[k]}}-\frac{t}{\sqrt n}\bar T_{(k)}
	\right )\\
	&& \hspace{0.5cm}+\frac{t^2}{2}\|\bar T_{(k)}\|_{L,k}^2+t\sqrt n \mathscr B_{n,k}+R_{n,k}(g,g_{0[k]})+\tilde r(g,g_{0[k]})\\
	&=& -\frac{n}{2}\left \| \log \frac{g e^{-\frac{t}{\sqrt n}\bar T_{(k)}}}{g_{0[k]}} \right \|_{L,k}^2
	+\sqrt nW_n \left (
	\log \frac{ge^{-\frac{t}{\sqrt n}\bar T_{(k)}}}{g_{0[k]}}
	\right )+\frac{t^2}{2}\|\bar T_{(k)}\|_{L,k}^2\\
	&&\hspace{0.5cm}+t\sqrt n \mathscr B_{n,k}+R_{n,k}(g,g_{0[k]})+\tilde r(g,g_{0[k]}).
\end{eqnarray*}
This is because the cross term in calculating the first term in the second line above is equal to the inner product term in the equation above it.

Let $g_{t,k}=ge^{-\frac{t}{\sqrt n}\bar T_{(k)}}/ G e^{-\frac{t}{\sqrt n}\bar T_{(k)}}$, the RHS  of the above equation can be written as
\begin{equation}\label{eq:power}
\frac {t^2}{2}\|\bar T_{(k)}\|_{L,k}^2 +l_n(g_{t,k})-l_n(g_{0[k]})+o(1).
\end{equation}
Substituting the corresponding terms on  the RHS of equation (\ref{eq:I}) by (\ref{eq:power}), we have that
\begin{equation}\label{eq:28}
E[ e^{t\sqrt n (T(g)-\hat T (g_{0[k]}))}  \mid \mathbb X_n, A_n] = e^{(t^2/2)\|\bar T_{(k)}\|_{L,k}^2+o(1)}\times
\frac{\int_{A_{n,k}} e^{l_n(g_{t,k})-l_n(g_{0[k]})} d\pi_k(g) }{\int_{A_{n,k}} e^{l_n(g)-l_n(g_{0[k]})} d\pi_k(g)}.
\end{equation}
Notice that the integration in the denominator of the second term is an expectation based on a Dirichlet distribution on $\omega$ as described in (\ref{eq:4.5}), and that $g_{t,k}=k\sum_{j=1}^k\zeta_j\Ind_{I_j}$, where
\begin{equation}\label{eq:6.6}
\zeta_j=\frac{\omega_j\gamma_j^{-1}}{\sum_{j=1}^k\omega_j\gamma_j^{-1}},
\end{equation}
with $\gamma_j=e^{t\bar T_j/\sqrt n}$ and $\bar T_j:= k\int_{I_j}\bar T_{(k)}$. Let $S_{\gamma^{-1}(\omega)}= \sum _{j=1}^k\omega_j\gamma_j^{-1}$, by (\ref{eq:6.6}) we have $S_{\gamma}^{-1}(\zeta)=S_{\gamma^{-1}}(\omega)$. Now using these notations,
\begin{eqnarray}\label{eq:trans}
\frac{\int_{A_{n,k}} e^{l_n(g_{t,k})-l_n(g_{0[k]})} d\pi_k(g) }{\int_{A_{n,k}} e^{l_n(g)-l_n(g_{0[k]})} d\pi_k(g)}
&=&\frac{\int_{A_{n,k}} e^{l_n(g_{t,k})-l_n(g_{0[k]})}  \prod_{j=1}^k \omega_j^{\alpha_{j,k}-1}/B(\alpha) d \omega }{\int_{A_{n,k}} e^{l_n(g)-l_n(g_{0[k]})} \prod_{j=1}^k \omega_j^{\alpha_{j,k}-1}/B(\alpha) d \omega} \\
& = &
\frac{\int_{A_{n,k}} e^{l_n
		(k\sum_{j=1}^k \frac{\omega_j\gamma_j^{-1}}{\sum_{j=1}^k\omega_j\gamma_j^{-1}} \Ind_{I_j})
		-l_n(g_{0[k]})}  \prod_{j=1}^k \omega_j^{\alpha_{j,k}-1} d \omega }{\int_{A_{n,k}} e^{l_n
		(k\sum_{j=1}^k\omega_j\Ind_{I_j})
		-l_n(g_{0[k]})} \prod_{j=1}^k \omega_j^{\alpha_{j,k}-1} d \omega}\nn\\
&=& \frac{\int_{A_{n,k}} e^{l_n
		(k\sum_{j=1}^k \zeta_j \Ind_{I_j})
		-l_n(g_{0[k]})} \Delta_{\zeta} \prod_{j=1}^k
	[\gamma_j \zeta_jS_{\gamma}^{-1}(\zeta)]^
	{\alpha_{j,k}-1}  d \zeta }
{\int_{A_{n,k}} e^{l_n
		(k\sum_{j=1}^k\omega_j\Ind_{I_j})
		-l_n(g_{0[k]})} \prod_{j=1}^k \omega_j^{\alpha_{j,k}-1} d \omega}, \nn
\end{eqnarray}
where  $\Delta_{\zeta}=S_{\gamma}^{-k}(\zeta)\prod_{j=1}^k\gamma_j$ is the Jacobian of the change of variable, $(\omega_1,\ldots,\omega_{k-1})\to (\zeta_1,\ldots,\zeta_{k-1})$,  which is given in Lemma 5 in supplement of \cite{castillo2015bernstein}, and $B(\alpha)=\prod_{i=1}^k \Gamma(\alpha_i)/\Gamma(\sum_{i=1}^k \alpha_i)$ is the constant for normalizing Dirichlet distribution.

Notice that over the set $A_{n,k}$,
\begin{eqnarray}\label{eq:trans2}
{\prod_{j=1}^k[\gamma_jS_{\gamma}^{-1}(\zeta)]^{\alpha_{j,k}-1}\Delta_{\zeta}}
&=&S_{\gamma}(\zeta)^{-\sum_{j=1}^{k}\alpha_{j,k}}\gamma_j ^{\sum_{j=1}^{k}\alpha_{j,k}} \nn \\
& = & S_{\gamma}(\zeta)^{-\sum_{j=1}^{k}\alpha_{j,k}}e^{t\sum_{j=1}^{k}a_{j,k}\bar T_j/\sqrt n}\nn\\
&=&e^{t\sum_{j=1}^{k}\alpha_{j,k}\bar T_j/\sqrt n} \left (
1-\frac{t}{\sqrt n}\int_0^1 \bar T_{(k)}(g-g_0)+O(n^{-1})
\right ) ^{\sum_{j=1}^{k}\alpha_{j,k}},
\end{eqnarray}
since
\[
S_{\gamma^{-1}}(\omega) =\int_0^1e^{-t\bar T_{(k)}(x)/\sqrt n} g_{[k]}(x)dx= 1-\frac{t}{\sqrt n}\int_0^1\bar T_{(k)}(g_{[k]}-g_0)+O(n^{-1})
\]
by Taylor's expansion.
Expression (\ref{eq:trans2}) converges to $1$ under the condition (\ref{eq:4.9}) and  hence expression (\ref{eq:trans})
converges to
\begin{equation}\label{eq:=1}
\frac{\int_{A_{n,k}} e^{l_n
		(k\sum_{j=1}^k \zeta_j \Ind_{I_j})
		-l_n(g_{0[k]})}  \prod_{j=1}^k
	\zeta_j^
	{\alpha_{j,k}-1}/B(\alpha_k)  d \zeta }
{\int_{A_{n,k}} e^{l_n
		(k\sum_{j=1}^k\omega_j\Ind_{I_j})
		-l_n(g_{0[k]})} \prod_{j=1}^k \omega_j^{\alpha_{j,k}-1}/B(\alpha_k) d \omega}
\end{equation}
since, when $\|\omega-\omega_0\|_1\leq M\sqrt{k\log n}/\sqrt n$,
$$
\|\zeta-\omega_0\|_1\leq \|  \omega-\omega_0 \|_1+ \| \omega-\zeta  \|_1 = \frac{M\sqrt{k\log n}+2|t|\|\tilde T\|_{\infty}}{\sqrt n}\leq (M+1)\frac{\sqrt{k\log n}}{\sqrt n}
$$
and vice versa,
when $\|\zeta-\omega_0\|_1\leq M\sqrt{k\log n}/\sqrt n$,
$$
\|\omega-\omega_0\|_1\leq \|  \omega-\zeta \|_1+ \| \omega_0-\zeta  \|_1 = \frac{M\sqrt{k\log n}+2|t|\|\tilde T\|_{\infty}}{\sqrt n}\leq (M+1)\frac{\sqrt{k\log n}}{\sqrt n}.
$$
Choosing $M$ such that
\begin{equation}\label{eq:end}
\pi\left [ \| \omega-\omega_0  \|_1\leq (M+1)\sqrt{k\log n}\mid \mathbb X_n, k\right ]
=1+o_p(1),
\end{equation}
expression (\ref{eq:=1}) equals to $1+o_p(1)$.
Notice that  $\|\bar T_{(k)} \|_{L,k} = \|\tilde T_{[k]}\|_L$, we have that
\begin{equation}\label{eq:6.7}
E^{\pi}\left[e^{t\sqrt n(T(g)-\hat T_k)}\mid \mathbb X_n,A_{n,k}\right]=e^{t^2\|\tilde T_{[k]}\|_L^2}\left(1+o_p(1)\right )
\end{equation}
which completes the proof for fixed $k$ case.

For random $k$ case, the proof will follow the same steps as the corresponding part in the proof for Theorem 4.2 in \cite{castillo2015bernstein}. For completeness, we briefly sketch  the proof here. Since $k$ is not fixed,  we will calculate $E^{\pi}[e^{t\sqrt n(T(f)-\hat T_k)}\mid \mathbb X_n ]$ on $B_n=\bigcup_{1\leq k\leq n}A_{n,k}\bigcap \{f=f_{\omega, k}, k\in \mathscr K_n\}$. Consider $\mathscr K_n$  a subset of $\{1,2,\ldots, n/\log^2 n\}$ such that $\pi(\mathscr K_n\mid \mathbb X_n)=1+o_p(1)$ by the concentration property (a) of the random histogram, we have that $\pi[B_n\mid \mathbb X_n]=1+o_p(1)$. We rewrite the L.H.S of (\ref{eq:6.7}) as $E^{\pi}[e^{t\sqrt n(T(f)-\hat T_k)}\mid \mathbb X_n,B_{n,k}]$ which is also equal to $e^{t^2\|\tilde T_{[k]}\|_L^2}(1+o_p(1))$. Notice that $o(1)$ in this expression is uniform in $k$. This is because it holds in the proof for deterministic case for any given $k<n$. Therefore,
\begin{eqnarray*}
	E^{\pi}\left[e^{t\sqrt n(T(f)-\hat T)}\mid \mathbb X_n,B_{n}\right]
	&=&\sum_{k\in\mathscr K_n}E^{\pi}\left[e^{t\sqrt n(T(f)-\hat T_k)+\hat T_k)-\hat T)}\mid \mathbb X_n,A_{n,k},k\right]\pi[k\mid \mathbb X_n]\\
	&=&(1+o(1))\sum_{k\in\mathscr K_n} e^{t^2V_k/2+t\sqrt n(\hat T_k-\hat T)}\pi[k\mid \mathbb X_n].
\end{eqnarray*}
Using (\ref{eq:4.10-2}) and (\ref{eq:4.12}) together with the continuous mapping theorem for the exponential function yields that the last display converges in probability to $e^{t^2V/2}$ as $n\to \infty$ which completes the proof.
\end{proof}

%

The following theorem shows that Method 2 is efficient, the proof which is to verify the conditions in above lemma are satisfied.

\begin{Theorem}\label{thm:4.2}
	Suppose $g_0\in \mathscr C^{\beta}$ with $\beta>0$. Let two priors $\pi_1$ and $\pi_2$ be defined by (\ref{eq:4.5})-(\ref{eq:4.6}) and prior on $k$ be either a Dirac mass at $k=K_n=n^{1/2}(\log n)^{-2}$ for $\pi_1$ or $k\sim \pi_k$ given by (\ref{eq:4.6}) for $\pi_2$. Then for all $\beta>1/2$, the BvM holds for $T(f)$ for both $\pi_1$ and $\pi_2$.
\end{Theorem}


\begin{proof}
%
%
%

For $T(f)$ such that  (\ref{eq:4.2}) is satisfied, condition (\ref{eq:4.11}) is satisfied obviously.



For equation (\ref{eq:4.10-2}), the empirical process $\mathbb G_n(\tilde T_{[k]}-\tilde{T})$ is controlled and will converge to $0$ by applying Lemma 19.33 in \cite{van2000asymptotic}.

Condition (\ref{eq:4.12}) is satisfied by Lemma \ref{lemma:balance} below.

Now we  show that  equation (\ref{eq:4.10-1}) holds: 

\begin{eqnarray*}
	\|\tilde T_f\|_L^2-\|\tilde T_{[k]}\|_L^2
	&\leq & \left|\int \dot s_{T(f)}(x)dx-\int \frac{\dot s_{T(f_{[k]}(x))}f(x)}{f_{[k]}(x)}dx  \right | \\
	 & \lesssim & \left|  \int \dot s_{T(f)}(x)f_{[k]}(x)-\dot s_{T(f[k])}(x)f(x)  \right |\\
	&= & \left |  \int \dot s_{T(f_{[k]})}(x)[f_{[k]}(x)-f(x)] \right | \\
	& \lesssim &  \int |f_{[k]}(x)-f(x)|dx.
\end{eqnarray*}
The last equality is based on conclusion (3) in Lemma 4 in \cite{Castillo13}, and the last inequality is due to the assumption that $\tilde T$ is bounded. Then the last term is controlled by $h(f,f_n)$, which completes the proof. \\
\end{proof}

\begin{Lemma}\label{lemma:balance}
	Under the same conditions as in Theorem \ref{thm:4.2}, 	equation (\ref{eq:4.12}) holds.
\end{Lemma}
\begin{proof} Since $\tilde T = \left ( -\left [\int  \ddot s_{T(g_0)}(x) g_0^{\frac{1}{2}}(x)dx\right ] ^{-1}   +a_t\right)
\frac{\dot s_{T(g_0)}(x)  } {2g_0^{\frac{1}{2}}(x)}
$, under the deterministic $k$-prior with $k=K_n=n^{1/2}(\log n)^{-2}$ and $\beta>1/2$,
$$
\left | \int (\tilde T-\tilde T_{[k]}) (g_0-g_{0[k]}) \right | \lesssim h^2(g_0,g_{0[k]})=o(1/\sqrt {n}).
$$
For the random $k$-prior, since we restrict $g$ to be bounded from above and below, so the Hellinger and $L^2$-distances considered are comparable. For given $k\in \mathcal K_n$, by definition there exists $g^*_k\in \mathcal H_k^1$ with $h(g_0, g^*_k)\leq M\epsilon_n(\beta)$, and hence,
$$
h^2(g_0,g_{0[k]})\lesssim \int (g_0-g_{0[k]})^2\leq \int (g_0-g^*_k)^2\lesssim h^2(g_0, g_k^*)\lesssim \epsilon_n^2(\beta),
$$
which completes the proof. \end{proof}

\section{Robustness properties}
In frequentist analysis, robustness is usually measured by the influence function and breakdown point of estimators.  These have been used to study robustness in minimum Hellinger distance estimators in \cite{Beran77} and in more general minimum disparity estimators in \cite{ParkBasu04} and \cite{Hooker11}.

In Bayesian inference, robustness is labeled ``outlier rejection'' and is studied under the framework of ``theory of conflict resolution''. There is large literature on this topic, {\em e.g.} \cite{deFinetti61}, \cite{Ohagan79}, and \cite{Ohagan90}.
While \cite{Ohagan90}'s results are only about symmetric distributions while \cite{desgagne07} gave corresponding results covering a wider class of distributions with tails in the general exponential power family. These results provided  a complete theory for the case of many observations and a single location parameter.

We examine the behavior of methods MHB and BMH under a mixture model for gross errors. Let $\delta _z$ denote the uniform density of the interval $(z-\epsilon, z+\epsilon)$, where $\epsilon>0$  is small, and let $f_{\theta, \alpha, z}=(1-\alpha)f_{\theta}+\alpha \delta_z$, where $\theta\in \Theta$ and $\alpha\in [0,1)]$ and $z$ is a real number. The density $f_{\theta, \alpha,z}$ models a situation, where $100(1-\alpha)\%$ observations are distributed from $f_{\theta}$ and $100\alpha\%$  of the observations are the gross errors located near $z$.

\begin{Theorem}
	For every $\alpha\in (0,1)$, every $\theta\in \Theta$, denote the mixture model for  gross errors by $f_{\theta,\alpha, z}$, we have that $lim_{z\to \infty} lim_{n\to \infty}T(g_n^*)=\theta$, under the assumptions of Theorem \ref{thm:2}; and that for the BMH method, $\pi(T(g) \mid \mathbb X_n)\to \phi(\theta, \|\tilde T_{f_{\theta, \alpha, z}}\|_L^2)$ in distribution as $n\to \infty$ and $z\to \infty$, where $\phi$ denotes the probability function of normal distribution, when conditions in Theorem \ref{thm:4.2} are satisfied.
\end{Theorem}

\begin{proof} By Theorem 7 in \cite{Beran77}, for functional $T$ as we defined and under the conditions in this theorem, we have that
$$
lim_{z\to \infty} T(f_{\theta,\alpha, z})=\theta.
$$
We also have that, for MHB, under conditions of Theorem \ref{thm:2},  $lim_{n\to \infty} T(g_n^*)\to T(f_{\theta,\alpha, z})$ in probability.  Combining the two results, $lim_{z\to \infty} lim_{n\to \infty}T(g_n^*)=\theta$, when the data is generated from a contaminated distribution as $f_{\theta, \alpha, z}$. Similarly, by Theorem \ref{thm:4.2}, we have that  $\pi(T(g) \mid \mathbb X_n)\to \phi(T(f_{\theta,\alpha,a}), \|\tilde T_{f_{\theta, \alpha, z}}\|_L^2)$ in distribution as $n\to \infty$, and  which converges to  $\phi(\theta, \|\tilde T_{f_{\theta, \alpha, z}}\|_L^2)$, as $z\to \infty$.
\end{proof}

\section{Demonstration}

We provide a demonstration of both BMH and MHB methods on two data sets: the classical Newcomb light speed data  (see \cite{Stigler1977}, \cite{BSP11}) in which 2 out of 66 values are clearly negative oultiers, and a bivariate simulation containing 10\% contamination in 2 asymmetric locations.

We have implemented the BMH and MHB methods using two Bayesian nonparametric priors:
\begin{enumerate}
\item the random histogram prior studied in this paper based on a fixed $k=100$ with the range naturally extend to the range of the observed data.  This is applied only to our first univariate example.

\item  the popular Dirichlet Process (DP) kernel mixture of the form
\begin{eqnarray*}
y_i \mid  \mu_i,\Sigma_i &\sim& N (\mu_{i}, \Sigma_i)  \\
(\mu_i, \Sigma_i) \mid G &\sim& G\\
G \mid \alpha, G_0 &\sim& DP(\alpha G_0)
\end{eqnarray*}
where, the baseline distribution is the conjugate normal-inverted Wishart,
$$
G_0= N(\mu\mid m_1, (1/k_0)\Sigma)IW(\Sigma \mid \nu_1, \psi_1).
$$
Note that when $y_i$'s are univariate observations, the inverse Wishart (IW) distribution reverts to being an inverse Gamma distribution.
To complete the model specification, independent hyperpriors are assumed
\begin{eqnarray*}
\alpha \mid   a_0, b_0 \sim Gamma(a_0,b_0)\\
m_1\mid m_2, s_2 \sim N(m_2,s_2)\\
k_0 \mid \tau_1,\tau_2 \sim Gamma(\tau_1/2, \tau_2/2)\\
\psi_1\mid \nu_2,\psi_2 \sim IW(\nu_2, \psi_2).
\end{eqnarray*}
\end{enumerate}
We obtain posteriors for both using BUGS. We have elected to use BUGS here as opposed to the package \texttt{DPpackage} within R despite the latter's rather efficient MCMC algorithms because our BMH method requires direct access to samples from the posterior distribution as opposed to the expected {\em a posteriori} estimate.  The R package \texttt{distrEx} is then used to construct the sampled density functions and calculated the Hellinger distance between the sampled densities from nonparametric model and the assumed normal distribution. The R package \texttt{optimx} is used to find the minima of the Hellinger distances. The time-cost of our methods are dominated by the optimization step, rather than in obtaining these samples.

We first apply BMH and MHB on the Simon Newcomb's measurements to measure the speed of light. The data contains 66 observations. For this example, we specify the parameters and hyper-parameters of the DPM as
$\alpha=1$,  $m_2=0, s_2=1000$, $\tau_1=1, \tau_2=100$, and $\nu_2=2, \psi_2=1$. We plot the data and a bivariate contour of the BMH posterior for both the mean and variance of the assumed normal in Figure \ref{fig:lightspeed}, where despite outliers, the BvM result is readily apparent.

Table \ref{table:light} summarizes these estimates. We report the estimated mean and variance with and without the obvious outliers as well as the same quantities estimated using both MHB and BMH methods with the last of these being the expected {\em a posteriori} estimates. Quantities in parentheses given the ``natural'' standard error for each quantity: likelihood estimates correspond to standard normal theory -- dividing the estimated standard error by $\sqrt{n}$, and BMH standard errors are obtained from the posterior distribution. For MHB, we used a bootstrap and note that while the computational cost involved in estimating MHB is significantly lower than BMH when obtaining a point estimate, the standard errors require and MCMC chain for each bootstrap, significantly raising the cost of obtaining these estimates.  We observe that both prior specifications result in parameter estimates that are identical to two decimal places and very close to those obtained after removing outliers.

\begin{figure}
\begin{center}
\begin{tabular}{cc}
\includegraphics[width=0.5\textwidth]{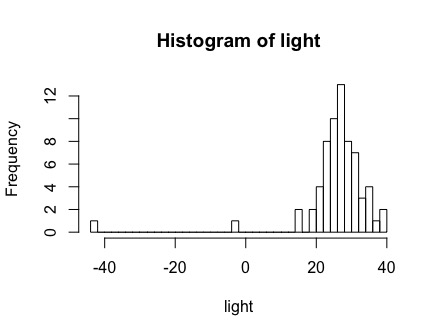} &
\includegraphics[width=0.5\textwidth]{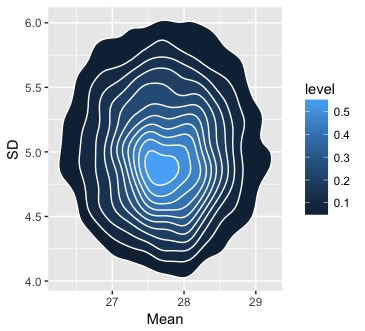}
\end{tabular}
\end{center}
	\caption{Left: Histogram of the light speed data; Right: bivariate contour plots of the posterior for the mean and variance of these data from the BMH method.} \label{fig:lightspeed}
\end{figure}

\begin{table}[htb]
\begin{center}
	\begin{tabular}{|c|c|c|c|c|}
		\hline
		& Direct Estimate&Without outliers&MHB&BMH\\\hline
		$\hat\mu$&  26.21 (1.32) &  27.75 (0.64) &  27.72  (0.64)&      27.73 (0.63) \\
		&   &        &  27.72  (0.64)&      27.73 (0.63) \\ \hline
		$\hat\sigma$& 10.75  (3.40) & 5.08  (0.46) &  5.07 (0.46) &  5.00 (0.47)\\
		& &  &  5.07 (0.46) &  5.00 (0.47)\\
		\hline
	\end{tabular}
\end{center}
\caption{Estimation results for Newcomb's light speed data. Direct Estimate refers to the standard mean and variance estimates, and ``Without outliers'' indicates the same estimates with outliers removed. The first row for each parameter gives the estimate under a Dirichlet Process prior and the second using a random histogram. Standard errors for each estimate are given in parentheses: these are from normal theory for the first two columns, via a bootstrap for MHB and from posterior samples for BMH.} \label{table:light}
\end{table}

To examine the practical implementation of methods that go beyond our theoretical results, we applied these methods to a simulated two-dimensional data set of 100 data points generated from a standard normal with two contamination distributions. Specifically, our data distribution comes from
\[
0.9 N\left( \left(\begin{array}{c} 10 \\ 5 \end{array} \right), \left( \begin{array}{cc} 3 & 1 \\ 1 & 5 \end{array} \right) \right) + 0.05 N \left( \left(\begin{array}{c} -2 \\ 5 \end{array} \right), \left( \begin{array}{cc} 0.5 & 0.1 \\ 0.1 & 0.5 \end{array} \right) \right) + 0.05 N \left( \left(\begin{array}{c} 10 \\ 14 \end{array} \right), \left( \begin{array}{cc} 0.4 & -0.1 \\ -0.1 & 0.4 \end{array} \right) \right)
\]
where exactly 5 points were generated from each of the second-two Gaussians. Our DP prior used the same hyper-parameters as above with the exception that $\Psi_1$ was obtained from the empirical variance of the (contaminated) data, and $(m_2,S_2)$ were extended to their 2-dimensional form as $\left ( (0,0)^T, diag(1000,1000)\right )$. Figure \ref{fig:2d.1} plots these data along with the posterior for the two means.  Figure \ref{fig:2d.2} provides posterior distributions for the components of the variance matrix.  Table \ref{table:2d} presents estimation results for the full data and those with the contaminating distributions removed as well as from the BMH method. Here we again observe that BMH gives results that are very close to those obtained using the uncontaminated data. There is some more irregularity in our estimates, particularly in Figure \ref{fig:2d.2} which we speculate is due to poor optimization. There is considerable scope to improve the numerics of minimum Hellinger distance methods more generally, but this is beyond the scope of this paper.


\begin{figure}[h]
	\begin{center}
		\begin{tabular}{cc}
			\includegraphics[width=0.5\textwidth]{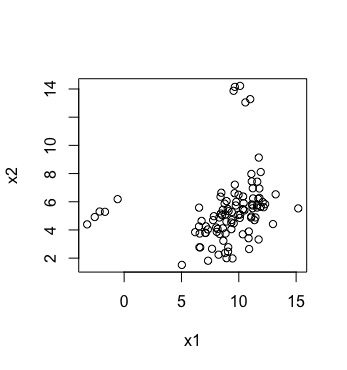}&
			\includegraphics[width=0.45\textwidth]{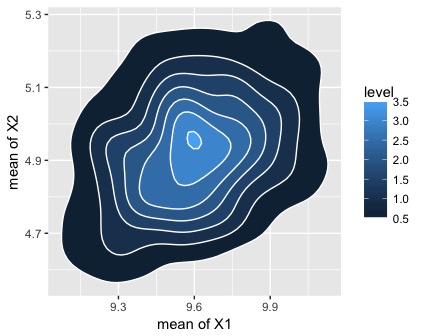}
		\end{tabular}
	\end{center}
	\caption{Left: simulated 2 dimensional normal example with two contamination components: Right: BMH posterior for the mean vector $(\mu_1,\mu_2)$.} \label{fig:2d.1}
\end{figure}


\begin{figure}
	\begin{center}
		\begin{tabular}{cc}
			\includegraphics[width=0.5\textwidth]{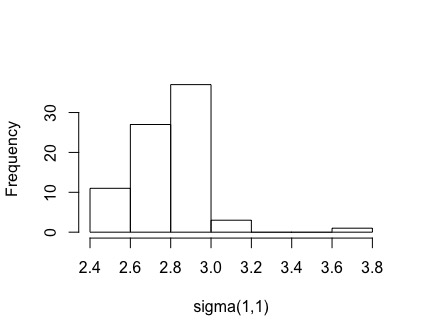}&
			\includegraphics[width=0.5\textwidth]{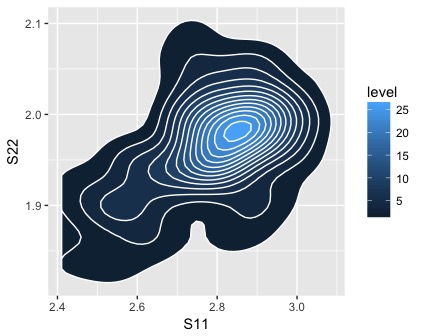}\\
			\includegraphics[width=0.5\textwidth]{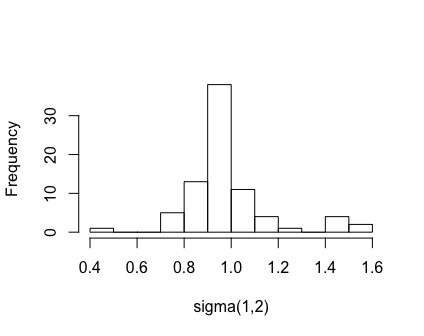}&
			\includegraphics[width=0.5\textwidth]{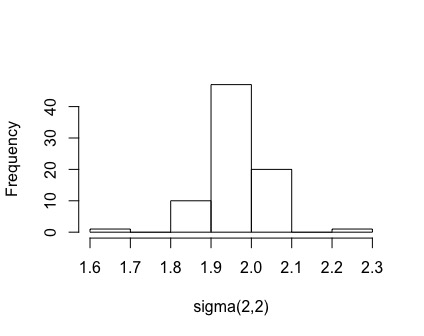}
		\end{tabular}
	\end{center}
	\caption{Posterior distributions for the elements of $\Sigma$ in the simulated bivariate normal example. } \label{fig:2d.2}
\end{figure}



\begin{table}[htb]
\begin{center}
	\begin{tabular}{|c|c|c|c|c|c|}
		\hline
		& $\mu_{01}$&$\mu_{02}$&$\Sigma_{11}$&$\Sigma_{12}$&$\Sigma_{22}$\\ \hline
		True&  10 &  5 &  3  &   1&2 \\  \hline
		 Contaminated data&9.07&5.36&9.76&1.67&5.80\\ \hline
	 Data with outliers removed &9.62 (0.13) &4.91 (0.11) &3.45 (0.13) &1.49 (0.13) & 2.29 (0.11)\\ \hline
		Estimated by BMH& 9.59 (0.27)   & 4.93 (0.19)       &  2.79  (0.18)  &   0.98 (0.18) & 1.97 (0.076) \\ \hline
	\end{tabular}
\end{center}
\caption{Estimation results for a contaminated bivariate normal. We provide generating estimates, the natural maximum likelihood estimates with and without outliers and the BMH estimates. Reported BMH estimates are expected {\em a posteriori} estimates with posterior standard errors given in paretheses. } \label{table:2d}
\end{table}


\section{Discussion}

This paper investigates the use of minimum Hellinger distance methods that replace kernel density estimates with Bayesian nonparametric models. We show that simply substituting the expected {\em a posteriori} estimator will reproduce the efficiency and robustness properties of the classical disparity methods first derived in \cite{Beran77}. Further, inducing a posterior distribution on $\theta$ through the posterior for $g$ results in a Bernstein von Mises theorem and a distributional robustness result.

There are multiple potential extensions of this work.  While we have focussed on the specific pairing of Hellinger distance and random histogram priors, both of these can be generalized. A more general class of disparities was examined in \cite{ParkBasu04} and we believe the extension of our methods to this class are straightforward.  More general Bayesian nonparametric priors are discussed in \cite{Ghosh2003} where the Dirichlet process prior has been particularly popular. Extensions to each of these priors will require separate analysis (e.g. \cite{wu08}).  Extensions of disparities to regression models were examined in \cite{Hooker16} using a conditional density estimate, where equivalent Bayesian nonparametrics are less well-developed. Other modeling domains such as time series may require multivariate density estimates, resulting in further challenges.

Our results are a counterpoint to the Bayesian extensions of Hellinger distance methods in \cite{Hooker11} where the kernel density was retained for $g_n$ but a prior was given for $\theta$ and the disparity treated as a log likelihood.  Combining both these approaches represents a fully Bayesian implementation of disparity methods and is an important direction of future research.

\end{document}